\documentclass[10pt]{amsart}
\usepackage{amssymb}

\usepackage{graphicx}
\usepackage{xcolor}
\usepackage{mathabx}

\newcommand*{\scirc}{\mathrel{\scalebox{0.8}{$\cdot$}}}

\newcommand{\res}{\upharpoonright}

\theoremstyle{plain}
\newtheorem{theorem}{Theorem}
\newtheorem{corollary}[theorem]{Corollary}
\newtheorem{lemma}[theorem]{Lemma}
\newtheorem{proposition}[theorem]{Proposition}
\newtheorem{definition}[theorem]{Definition}

\theoremstyle{remark}

\newtheorem*{claim}{Claim}

\begin{document}

\title{Finite Ramsey theory through category theory}

\author{S{\l}awomir Solecki}

\address{Department of Mathematics\\
Cornell University\\
Ithaca, NY 14853}

\email{ssolecki@cornell.edu}

\urladdr{https://e.math.cornell.edu/people/ssolecki/}

\thanks{Research supported by NSF grant DMS-1954069.}



\begin{abstract} We present a new, category theoretic point of view on finite Ramsey theory. Our aims are as follows:  
\begin{enumerate}
\item[---] to define the 
category theoretic notions needed for the development of finite Ramsey Theory, 

\item[---] to state, in terms of these notions, the general fundamental Ramsey results (of
which various concrete Ramsey results are special cases), and 

\item[---] to give self-contained proofs within the category theoretic framework of these general results. 
\end{enumerate} 
We also provide some concrete illustrations of the general method. 
\end{abstract}

\maketitle


\tableofcontents

\section{Introduction}

For $r\in {\mathbb N}$, by an {\bf $r$-coloring} of a set $X$ we understand a function on $X$ with at most $r$ values. If $C$ is a category, by ${\rm ob}(C)$ we denote the class 
of all objects of $C$; for $a,b\in {\rm ob}(C)$, ${\rm hom}(a,b)$ stands for the class of all morphisms in $C$ from $a$ to $b$.

It has been known for some time, at least since the early 1970s, that Ramsey theoretic statements are naturally expressed in the language of category theory. 
Given a category $C$,  
an object $a\in {\rm ob}(C)$ is said to have the {\bf Ramsey property} if for each object $b\in {\rm ob}(C)$ and $r\in {\mathbb N}$, 
there is $c\in {\rm ob}(C)$ such that for each $r$-coloring of ${\rm hom}(a, c)$, 
there is $g\in {\rm hom}(b,c)$ that makes the set 
\[
g\cdot {\rm hom}(a,b)= \{ g\cdot f\mid f\in {\rm hom}(a,b)\}
\] 
monochromatic. 
One then calls $C$ a {\bf Ramsey category} if each of its objects has the Ramsey property. 
It is also known that the following notion is a useful refinement of the Ramsey property. For $a,b \in {\rm ob}(C)$, 
the {\bf Ramsey degree} of the pair $a, b$ is the smallest natural number $k$, if such a number exists, such that for each 
$r\in {\mathbb N}$, there 
exists $c\in {\rm ob}(C)$ with the property that, for each $r$-coloring $\chi$ of $\hom(a,c)$, there exists $g\in \hom(b,c)$ with 
$\chi$ attaining at most $k$ values on $g\cdot \hom(a,b)$. If such a number $k$ does not exist, we say that the Ramsey 
degree of the pair $a,b$ is $\infty$. We write 
\begin{equation}\label{E:rdt}
{\rm rd}(a,b)
\end{equation}
for the Ramsey degree of $a,b$. More studied, and often more useful, notion is that of Ramsey degree of a single object. {\bf Ramsey degree of} $a$ 
is defined by 
\begin{equation}\label{E:rdo}
{\rm rd}(a)= \sup_{b\in {\rm ob}(C)} {\rm rd}(a,b). 
\end{equation} 
Having the Ramsey property is expressible in terms of the Ramsey degree---$a$ has the Ramsey property precisely when ${\rm rd}(a)= 1$. 
This observation leads to an important easing of the condition of $C$ being Ramsey to the condition of 
$C$ having {\bf finite Ramsey degrees}, which asserts that, for each $a\in {\rm ob}(C)$, ${\rm rd}(a)<\infty$. 
The importance of finiteness of Ramsey 
degree stems partly from it being a refinement of the Ramsey property and partly, and more significantly, from its relevance to 
topological dynamics as indicated in \cite{KPT} and, especially, in \cite{Zuc}.

Going beyond just formulating Ramsey theoretic notions, 
several papers used the language of category theory to carry out proofs of finite Ramsey theoretic statements; 
see, for example, \cite{GLR}, \cite{SJ}, \cite{Lee}, \cite{Mas2}, and \cite{Mas}. This was usually done by identifying 
a category $A$ that was known to be Ramsey or to have a related property, and then transferring Ramseyness or the related property from $A$ 
to another category $B$ by finding 
an, often subtle, connection between $A$ and $B$. Another distinct approach was presented by Leeb in \cite{Lee}. 
In a number of specific categories, he verified 
certain identities, called by him Pascal identities, and then working separately in each of these categories, but using similarly structured 
arguments, he showed that the categories are Ramsey, from which various classical Ramsey theorems followed. 
In \cite{SJ}, certain Ramsey theoretic constructions were revealed to be canonical category theory constructions. 
In his paper \cite{Gro}, Gromov advocated for a broad use of category theory in Ramsey theory.

In present paper, expanding and simplifying the author's approach of \cite{Sol} and following the spirit of \cite{Gro}, 
we present a way of seeing finite Ramsey theory in the category theoretic terms that is global, in the sense that 
the whole theory is developed from scratch in the category theoretic framework as opposed to transferring specific Ramsey results between categories. 
We change the usual perspective of considering categories as basic to considering functors as fundamental to the development. 
More precisely, we formulate a general pigeonhole principle (P) for functors between categories. On our view, this is the main notion of finite Ramsey theory, and 
the theory is concerned with proving (P) for various functors. To this end,  
we show that (P) persists over several natural operations applied to functors. 
Further, we formulate a localized version (FP) of (P) and show that, for essentially all relevant functors, (FP) implies (P). 
The advantage of having this implication resides in a relative ease of proving the localized pigeonhole principle (FP) for many functors in comparison 
with proving (P) itself. 
These general theorems allow us to establish results giving an upper bound on Ramsey degrees. As explained above, such upper estimates are generalizations of 
the statement that the category is Ramsey thereby making it possible to deduce various concrete Ramsey theorems. 

It may be worth emphasizing that the point we are making is not so much that the concrete Ramsey theorems 
can be derived using our methods, it is more that they are particular cases (for concrete functors) of our 
results. We give several examples of theorems that can be seen as such---Ramsey's original theorem, the product Ramsey theorem, 
the Hales--Jewett theorem, Fouch{\'e}'s 
Ramsey theorem for trees \cite{Fou}. They should be viewed merely as illustrations since many other theorems, for example, the Ramsey theorems for trees due 
to various mathematicians that are surveyed in \cite{Sol2} and treated there with the methods of \cite{Sol}, the Ramsey theorems considered in \cite{Sol}, 
for example, the dual Ramsey theorem of Graham--Rothschild, and the dual Ramsey theorem for trees from \cite{Solt}, 
can all be seen as, in essence, particular cases of the general theorems of this paper. 

The present paper is a contribution to the efforts aimed at unifying Ramsey theoretic results, as in \cite{HN} for structural Ramsey theory and in \cite{Tod} for infinite dimensional Ramsey theory. 
Our approach here builds on \cite{Sol}. The main advances with respect to that paper consist of the use of categories and functors 
instead of various types of ad hoc structures (in particular, eliminating partial or linear orders from the general structures), weakening of 
the localized pigeonhole principle (from $\rm (LP)$ in \cite{Sol} to $\rm (FP)$ here), obtaining upper estimates on the Ramsey degree rather than just Ramsey statements, 
and an overall substantial simplification of the presentation.

The following conventions will be used throughout. By ${\mathbb N}$ we understand the set of all natural numbers including $0$. 
For $m, n\in {\mathbb Z}$, we write 
\[
[m, n] = \{ i\in {\mathbb Z}\mid m\leq i\leq n\}, 
\]
For $n\in {\mathbb N}$ and $m=1$, we shorten the above piece of notation to 
\[
[n]= [1, n]. 
\]
In particular, $[0]=\emptyset$. The cardinality of a set $x$ will be denoted by $| x|$. So, if $x$ is finite, then $|x|\in {\mathbb N}$. 
For a functor $\gamma$ defined on a category $C$ and for an object $a$ of $C$ and a morphism $f$ of $C$, we often write 
\[
\gamma a\;\hbox{ and }\;\gamma f
\]
for $\gamma(a)$ and $\gamma(f)$.

\smallskip

I would like to thank Sebastian Junge, whose remarks improved the presentation of the material in this paper.

\section{Condition (P) and frank functors}

\subsection{Formulation of condition (P)} 

We regard the following statement as the fundamental pigeonhole principle for a functor $\delta\colon C\to D$ between the categories $C$ and $D$. 
In a nutshell, it says that, in a suitable sense, $\delta$ 
controls colorings. 

\begin{definition}\label{D:defp} 
Let $\delta$ be a functor defined on a category $C$, and let $a,b\in {\rm ob}(C)$. 
We declare $\delta$ to {\bf fulfill {\rm (P)} at} $a, b$ if 
for each $r\in {\mathbb N}$, there exists $c\in {\rm ob}(C)$ such that, for each $r$-coloring $\chi$ of $\hom(a,c)$, there
exists $g\in \hom(b,c)$ with 
\[
\delta f_1 = \delta f_2\Longrightarrow \chi(g\cdot f_1)=\chi(g\cdot f_2),
\]
for all $f_1, f_2\in \hom(a,b)$.
\end{definition}

For ease of phrasing, we adopt the following conventions. 
We say that $\delta$ {\bf fulfills {\rm (P)} at $a\in {\rm ob}(C)$} if it fulfills {\rm (P)} at $a,b$ for all $b\in {\rm ob}(C)$. 
We say simply that $\delta$ {\bf fulfills {\rm (P)}} if it fulfills {\rm (P)} at all $a,b\in {\rm ob}(C)$. 
Note that the condition of fulfilling (P) at $a$ is stronger than fulfilling (P) at $a,b$. We will be mostly interested in this stronger condition, but using 
the weaker condition is somewhat easier and permits us to make more precise statements in some situations.

\medskip

The fundamental connection of property (P) with Ramsey theoretic notions of Ramsey degrees (recall \eqref{E:rdt} and \eqref{E:rdo} here) 
goes through the following proposition, which improves the trivial bound 
${\rm rd}(a,b)\leq |{\rm  hom}(a,b)|$.

\begin{proposition}\label{P:bas} 
Let $\Delta$ be a family of functors whose domains are all equal to $C$. 
\begin{enumerate} 
\item[(i)] If each $\delta\in \Delta$ fulfills {\rm (P)} at $a, b\in {\rm ob}(C)$, then 
\[
{\rm rd}(a,b)\leq \min_{\delta\in \Delta} \big|\delta\big({\rm  hom}(a,b)\big)\big|. 
\]

\item[(ii)] If each $\delta\in \Delta$ fulfills {\rm (P)} at $a\in {\rm ob}(C)$, then 
\[
{\rm rd}(a)\leq \sup_{b\in {\rm ob}(C)}\min_{\delta\in \Delta} \big|\delta\big({\rm  hom}(a,b)\big)\big|. 
\]
\end{enumerate} 
\end{proposition}

\begin{proof} By the definition of ${\rm rd}(a)$ in \eqref{E:rdo}, it suffices to show (i). 
This amounts to proving that 
given $r\in {\mathbb N}$, there exists $c$ such that for each $r$-coloring of $\hom(a,c)$, there exists $g\in \hom(b,c)$ 
with the number of colors attained on $g\cdot \hom(a,b)$ bounded by $|\delta\big( \hom(a,b)\big)|$. This statement 
is an immediate consequence of $\delta$ fulfilling (P) at $a$. 
\end{proof}

\subsection{Definition of frank functors}

When dealing with property (P) for a functor $\delta$, sometimes we will need to make an additional surjectivity assumption on $\delta$. This surjectivity assumption is 
critical and seems interesting enough to isolate it here.

\begin{definition}
A functor $\delta\colon C\to D$ between two categories $C$ and $D$ is called {\bf frank} if, 
for all $a\in {\rm ob}(C)$ and $b'\in {\rm ob}(D)$, 
there exists $b\in {\rm ob}(C)$ with 
\begin{equation}\notag
\delta (b)= b'\;\hbox{ and }\; \delta\big(\hom(a,b)\big) =\hom(\delta a, \delta b). 
\end{equation} 
\end{definition}

In the second equality in the definition above, the inclusion $\delta\big(\hom(a,b)\big) \subseteq \hom(\delta a, \delta b)$ follows just from $\delta$ being a functor. 
Therefore, the point of the equality is that $\delta$ is surjective as a function from $\hom(a,b)$ to $\hom(\delta a, \delta b)$.

\begin{lemma}\label{L:trz} 
\begin{enumerate}
\item[(i)] The identity functor is frank. 

\item[(ii)] The composition of frank functors is a frank functor. 

\item[(iii)] If $\delta\colon C\to D$ is a frank functor, then, for all $d_1, d_2\in {\rm ob}(D)$, there are $c_1, c_2\in {\rm ob}(C)$ such that 
\[
\delta(c_1)=d_1,\, \delta(c_2)= d_2,\hbox{ and }\, \delta\big(\hom(c_1, c_2)\big) =\hom(d_1, d_2). 
\] 
\end{enumerate}
\end{lemma} 

\begin{proof} Points (i) and (ii) are almost immediate, and we leave checking them to the reader. To see (iii), given $d_1$, use frankness of $\delta$ 
to find $c_1$ with $\delta(c_1)= d_1$. Now use frankness of $\delta$ again to find $c_2$ with $\delta(c_2)= d_2$ and $\delta\big(\hom(c_1, c_2)\big) =\hom(d_1, d_2)$. 
\end{proof}

\subsection{Unifying assumptions} 

Without harming applicability in finite Ramsey theory of the general theorems presented below, one may always make the following unifying assumptions: 
\begin{enumerate} 
\item[---] {\bf functors are frank} 

\item[---] {\bf categories are such that $\hom(a,b)$ is finite for all objects $a,b$}. 
\end{enumerate}
All the general theorems proved below hold under these assumptions. 
Of course, when stating the theorems, we make assumptions that are appropriate (minimal) for each theorem.

\subsection{Examples---frank functors fulfilling (P)}\label{Su:exp}
We describe here some examples of frank functors for with property (P). 

\smallskip

{\bf 1.} Let $C$ be a category. The identity functor $C\to C$ is frank and fulfills (P) at each pair of objects of $C$.

\smallskip

{\bf 2.} We define here a category $P$ and a frank functor $\partial_P\colon P\to P$ that will play an important auxiliary role in proving the Hales--Jewett theorem. 
Fulfilling of condition (P) by $\partial_P$ is the usual pigeonhole principle. 

\noindent Objects of $P$ are 
\begin{enumerate}
\item[---] pairs $(k,i)$, where $k\in {\mathbb N}$, $k\geq1$ if $i\in \{ 0,  2\}$, and $k\geq2$ if $i=1$. 
\end{enumerate}
Morphisms of $P$ from an object $(k,i)$ to an object $(k',i')$ will be certain functions from $[k']$ to $[k]$, whose 
nature will depend on the second coordinates $i$ and $i'$. We define the morphisms as follows: 
\begin{enumerate}
\item[---] $p\in {\rm hom}\big((l,2), (m,2)\big)$, if $p\colon [m]\to [l]$ is a surjection with $p(i)\leq p(i+1)\leq p(i)+1$ for $i\in [m-1]$; 
\item[---] $x\in {\rm hom}\big((k_1,1), (l,2)\big)$, if $x\colon [l]\to [k_1]$ and there are $1\leq a< a+1< b\leq l$ such that $x$ is 
constant on the intervals $[1,a]$, $[a+1, b-1]$, and $[b,l]$ and 
\[
x(1)= k_1\; \hbox{ and }\; x(l)=k_1-1;
\]
\item[---] $x\in {\rm hom}\big((k_0,0), (l,2)\big)$, if $x\colon [l]\to [k_0]$ and there are $1\leq a< a+1< b\leq l$ such that $x$ is 
constant on the intervals $[1,a]$, $[a+1, b-1]$, and $[b,l]$ and
\[
x(1)= x(l)=k_0; 
\]
\item[---] there are no other morphisms except for identities. 
\end{enumerate}
Composition of morphisms in $P$ will be the composition of functions taken with reverse order: 
\begin{enumerate}
\item[---]  for $p \in {\rm hom}\big((l,2), (m,2)\big)$ and $x\in {\rm hom}\big((k,i), (l,2)\big)$, with $i\in \{0,1\}$, let  
\[
p\cdot x = x\circ p. 
\]

\item[---]  for morphisms $p \in {\rm hom}\big((l,2), (m,2)\big)$ and $q \in {\rm hom}\big((m,2), (n,2)\big)$, let  
\[
q\cdot p = p\circ q. 
\]

\end{enumerate} 
The functor $\partial_P\colon P\to P$ acts non-trivially only on objects of the form $(k,1)$ and on morphisms in ${\rm hom}\big( (k,1), (l,2)\big)$, on which it lowers 
the top value by $1$. So the functor $\partial_P$ is defined by: 
\begin{enumerate}
\item[---] $\partial_P (k,i) = (k,i)$, for $i\not= 1$, and $\partial_P (k,1)= (k-1, 0)$;

\item[---] $\partial_P p = p$, for $p\in {\rm hom}\big((l,2), (m,2)\big)$;

\item[---] $\partial_P x = x$, for $x\in {\rm hom}\big((k_0,0), (l,2)\big)$; 

\item[---] $\partial_P x = \min (x, k_1-1)$, for $x\in {\rm hom}\big((k_1,1), (l,2)\big)$. 
\end{enumerate}

\begin{lemma}\label{L:pfr}
\begin{enumerate}
\item[(i)] $\partial_P\colon P\to P$ is a functor and it is frank. 

\item[(ii)] $\partial_P$ fulfills (P). 
\end{enumerate}
\end{lemma}

\begin{proof} Point (i) is straightforward, and we leave checking it to the reader. 

To see (ii), we check (P) at $(k, 1)$, $(l, 2)$, which is the only not entirely trivial case. 
After fixing $r\in {\mathbb N}$, we need to find an object $(m,2)$ for which the conclusion of (P) holds. It is good to keep in mind that all we are doing 
is proving a version of the standard pigeonhole principle. 

Put $m=(l-1)r+2$. Let $\chi$ be an $r$-coloring of 
${\rm hom}\big((k_1, 1), (m,2)\big)$, that is, an $r$-coloring of the set of all $x\colon [m]\to [k_1]$, for which there exist 
$1\leq a_x< a_x+1< b_x\leq m$ such that $x$ is 
constant on the intervals $[1,a_x]$, $[a_x+1, b_x-1]$, and $[b_x,m]$, and $x(1)= k_1-1$ and $x(m)= k_1$. Consider the $r$-coloring 
$\chi'$ of $[m-1]$ given by 
\[
\chi'(j) = \chi(x_j),
\]
where $x_j\colon [m]\to [k_1]$, with $j\in [m-1]$, is such that 
\[
x_j\res [j] = k_1-1\;\hbox{ and }\; x_j\res [j+1,m] = k_1. 
\]
Note that $x_j\in {\rm hom}\big((k_1, 1), (m,2)\big)$, so $\chi(x_j)$ is defined. By our choice of $m$, $\chi'$ is constant on a subset of 
$[m-1]$ of size $l$. So there exists $p\colon [m]\to [l]$, a surjection with $p(j)\leq p(j+1)\leq p(j)+1$ for $j\in [m-1]$, that is, 
$p\in {\rm hom}\big( (l, 2), (m,2)\big)$, and such that $\chi'$ is constant on the set 
\begin{equation}\notag
J= \{ j\in [m-1]\mid p(j)< p(j+1)\}.
\end{equation}
The above condition on $\chi'$ means that $\chi(x_j)$ is constant as $j$ varies over $J$.  

We claim that this $p$ works. Indeed, let $x, x'\in {\rm hom}\big( (k_1,1), (l,2)\big)$ be such that $\partial_P x = \partial_P x'$. Then either $x=x'$ or 
there are $i, i'\in [l-1]$ such that 
\[
\begin{split}
x\res [i] = k_1-1,&\; x\res [i+1,m] = k_1,\\ 
x'\res [i'] = k_1-1,&\; x'\res [i'+1,m] = k_1. 
\end{split}
\]
In the first case, clearly $\chi(x\circ p) = \chi(x'\circ p)$. In the second case, for $j,j'\in J$ specified by 
\[
i=p(j)<p(j+1)\;\hbox{ and }\; i'=p(j')<p(j'+1), 
\]
we have $x\circ p = x_j$ and $x'\circ p =x_{j'}$. 
Thus, we get 
\[
\chi(x\circ p) = \chi(x_j)= \chi(x_{j'})= \chi(x'\circ p),
\]
as required. 
\end{proof}

\smallskip 

{\bf 3.} The following example does fulfill condition (P), but we postpone its verification till we have a general result from which it will follow readily. The 
example will serve to prove the standard Ramsey theorem. 

We define a category $R$ and a functor $\partial_R \colon R\to R$. The objects of the underlying category will be natural numbers $n$ and morphisms will be, 
essentially,  
subsets $x$ of $[n]$. For technical reasons, for such a set $x$, we will need to remember which $[n]$ it is designated to be a subset of; 
so the morphisms will actually be pairs $(x,n)$. 
For $n\in {\mathbb N}$, we write 
\begin{equation}\label{E:dotmin}
n\dotdiv 1 = \max( n-1, 0). 
\end{equation} 

Objects of $R$ are: 
\begin{enumerate}
\item[---] $n$, for $n\in {\mathbb N}$. 
\end{enumerate}

Morphisms of $R$ are described as follows: 
\begin{enumerate}
\item[---] $(x,n)\in {\rm hom}(m, n)$, for $x\subseteq [n]$ and $|x|= m$.
\end{enumerate}

Composition in $R$ is defined by the following rule:  
\begin{enumerate}
\item[---]  for morphisms $(x,m) \in {\rm hom}(l,m)$ and $(y,n) \in {\rm hom}(m,n)$, let  
\[
(y,n) \cdot (x,m) = (f_y(x), n),
\]
where $f_y\colon [m]\to y$ is the unique increasing bijection. 
\end{enumerate} 

The functor $\partial_R\colon R\to R$ is defined using the notation set up by \eqref{E:dotmin}:
\begin{enumerate}
\item[---] $\partial_R n = n\dotdiv 1$; 

\item[---] $\partial_R (x,n) = (x\setminus \{ \max x\}, n\dotdiv 1)$, for $(x, n)\in {\rm hom}(m,n)$ with $x\not=\emptyset$; 

\item[---] $\partial_R (\emptyset, n) = (\emptyset, n\dotdiv 1)$, for $(\emptyset, n)\in {\rm hom}(0,n)$.
\end{enumerate}

\section{Propagating condition (P)}

In this section, we prove three theorems that let us transfer property (P) from one functor to another.

\subsection{Composition}

The following theorem asserts that, under appropriate assumptions, property (P) is preserved under composition of functors. 

\begin{theorem}\label{T:semig}
Let $\gamma \colon C\to D$ and $\delta\colon D\to E$ be functors with $\gamma$ being frank. Let $a,b\in {\rm ob}(C)$. 
If $\gamma$ fulfills {\rm (P)} at $a$ and $\delta$ fulfills {\rm (P)} at $\gamma(a), \gamma(b)$, then $\delta\circ\gamma$ fulfills {\rm (P)} at $a, b$.  
\end{theorem} 

\begin{proof} We write $\delta\gamma$ for $\delta\circ\gamma$. 

Fix $a, b\in {\rm ob}(C)$ with the aim to show that $\delta\gamma$ has (P) at $a,b$. In order to do this, let $r\in {\mathbb N}$. The objects 
$a,b$ and the natural number $r$ will remain fixed for the rest of this proof. 

\begin{claim} There is
$c\in {\rm ob}(C)$ such that for every $r$-coloring $\chi$ of $\hom({\gamma}a,{\gamma} c)$, 
there is $g\in \hom(b,c)$ such that for $f_1, f_2\in \hom(a,b)$
\begin{equation}\label{E:tro}
\delta {\gamma} f_1 = \delta {\gamma} f_2 \Longrightarrow \chi\bigl({\gamma} (g\scirc f_1)\bigr) = \chi\bigl({\gamma} (g\scirc f_2)\bigr).
\end{equation}
\end{claim}

\noindent{\em Proof of Claim.} 
Since $\delta$ fulfills (P) at $\gamma a$ and $\gamma b$, we can find $d\in {\rm ob}(D)$ such that 
for every $r$-coloring $\chi$ of ${\rm hom}(\gamma a, d)$, 
there is $g'\in \hom(\gamma b,d)$ such that, for $h_1,h_2\in \hom(\gamma a, \gamma b)$,  
we have 
\begin{equation}\label{E:cra}
\delta h_1 = \delta h_2 \Longrightarrow \chi\bigl(g'\scirc h_1\bigr) = \chi\bigl( g'\scirc h_2\bigr).
\end{equation}
By frankness of the functor $\gamma$, there is $c\in {\rm ob}(C)$ such that $d=\gamma c$ and, 
for every $g'\in \hom(\gamma b,d)$, there is $g\in \hom(b, c)$ with $g'=\gamma g$. 
We claim that this $c$ makes the conclusion of the claim true. 

Let $\chi$ be a $r$-coloring of $\hom({\gamma}a, {\gamma} c)= \hom({\gamma} a, d)$. 
By our choice of $d$, there
is $g'\in \hom(\gamma b, d)$ such that, for $h_1,h_2\in \hom(\gamma a, \gamma b)$, condition \eqref{E:cra} holds. 
Let $g\in \hom(b, c)$ be such that $g'=\gamma g$. In order to check condition \eqref{E:tro}, 
fix $f_1,f_2\in \hom(a,b)$ with 
\begin{equation}\label{E:astr}
\delta {\gamma} f_1 = \delta {\gamma} f_2. 
\end{equation} 
Note that $\gamma f_1,\, \gamma f_2\in \hom(\gamma a, \gamma b)$, and therefore 
by \eqref{E:cra} and \eqref{E:astr}, we have 
\[
\chi\bigl(g'\scirc (\gamma f_1)\bigr) = \chi\bigl(g'\scirc (\gamma f_2)\bigr).
\]
Since 
\[
g'\scirc (\gamma f_1) = (\gamma g)\scirc (\gamma f_1) = \gamma(g\scirc f_1)\;\hbox{ and }\; 
g'\scirc ( \gamma f_2) = (\gamma g)\scirc (\gamma f_2)=\gamma(g\scirc f_2), 
\]
it follows that $\chi\bigl({\gamma}(g\scirc f_1)\bigr) = \chi\bigl({\gamma}(g\scirc f_2)\bigr)$, which gives \eqref{E:tro} and the claim.

\smallskip
Now, we prove the conclusion of the theorem from the claim. 
We are seeking $c\in {\rm ob}(C)$ with the following property: for each $r$-coloring $\chi$ of $\hom(a,c)$ 
there exists $g\in \hom(b,c)$ such that, for $f_1, f_2\in \hom(a,b)$,
\begin{equation}\label{E:alth} 
\delta {\gamma} f_1 = \delta {\gamma} f_2 \Longrightarrow \chi(g\scirc f_1)= \chi(g\scirc f_2).
\end{equation}
We apply the claim to obtaining $c'\in {\rm ob}(C)$. Next, recall that we assume that $\gamma$ fulfills (P) at $a$, so it fulfills (P) at $a, c'$, which, for the given $r$, yields 
$c\in {\rm ob}(C)$. 
We claim that this $c$ works.

Let $\chi$ be a $r$-coloring of $\hom(a, c)$. By the choice of $c$ there
exists $g'\in \hom(c',c)$ such that for $f_1, f_2\in \hom(a,b)$ and $h_1, h_2\in \hom(b,c')$,
\begin{equation}\label{E:mirac}
{\gamma} (h_1\scirc f_1) = {\gamma} (h_2\scirc f_2)
\Longrightarrow \chi\bigl(g'\scirc (h_1\scirc f_1)\bigr) = \chi\bigl(g'\scirc (h_2\scirc f_2)\bigr).
\end{equation}
We define a $r$-coloring ${\bar \chi}$ on $\hom({\gamma} a, {\gamma} c')$. First we specify $\bar\chi$ of the subset 
\[
\{ \gamma(h\cdot f)\mid  f\in {\rm hom}(a,b),\, h\in {\rm hom}(b,c')\}
\]
of $\hom({\gamma} a, {\gamma} c')$. So, for $f\in \hom(a,b)$ and $h\in \hom(b, c')$, let 
\begin{equation}\label{E:mircont}
{\bar \chi}\bigl({\gamma} (h\cdot f)\bigr) = \chi\bigl(g'\scirc h\scirc f\bigr).
\end{equation}
The function ${\bar \chi}$ is well-defined by \eqref{E:mirac}. Now we extend $\bar\chi$ to an $r$-coloring of 
the whole set $\hom({\gamma} a, {\gamma} c')$
in an arbitrary way. We denote this extension again by $\bar\chi$. By our
choice of $c'$ from Claim, there exists $g''\in \hom(b,c')$ such that, for $f_1, f_2\in \hom(a,b)$, 
\begin{equation}\label{E:miren}
\delta {\gamma} f_1 = \delta {\gamma} f_2 \Longrightarrow
{\bar \chi}\bigl({\gamma}(g''\scirc f_1)\bigr) = {\bar \chi}\bigl({\gamma}(g''\scirc f_2)\bigr).
\end{equation}
Combining \eqref{E:miren} with \eqref{E:mircont}, we see that, for $f_1, f_2\in \hom(a,b)$, 
\begin{equation}\notag
\delta {\gamma} f_1 = \delta {\gamma} f_2 \Longrightarrow \chi\bigl((g'\scirc g'')\scirc f_1\bigr)
= \chi\bigl((g'\scirc g'')\scirc f_2\bigr).
\end{equation}
Thus, $g= g'\scirc g''\in \hom(b,c)$ 
is as required by \eqref{E:alth}. 
\end{proof}

The following corollary improves the estimate from Proposition~\ref{P:bas}. 
For a family $\Delta$ of endofunctors of a category $C$, by $\langle \Delta\rangle$ we denote the semigroup generated by $\Delta$ 
using composition, that is,  
\[
\langle \Delta\rangle = \{ \overline{\delta} \mid \overline{\delta} =\delta_1\circ\cdots \circ \delta_n,\hbox{ for }\delta_1, \dots , \delta_n\in \Delta\}. 
\]

\begin{corollary}\label{C:pest} 
Let $\Delta$ be a family of endofunctors of $C$ with each endofunctor in $\Delta$ being frank and fulfilling {\rm (P)}. 
\begin{enumerate}
\item[(i)] For each $a, b\in {\rm ob}(C)$, we have 
\[
{\rm rd}(a,b) \leq \min \{ \big|\overline{\delta}\big({\rm hom}(a,b)\big)\big|\mid \overline{\delta}\in \langle \Delta\rangle\}. 
\]

\item[(ii)] For each $a\in {\rm ob}(C)$, 
\[
{\rm rd}(a) \leq \sup_{b\in {\rm ob}(C)} \min \{ \big|\overline{\delta}\big({\rm hom}(a,b)\big)\big|\mid \overline{\delta}\in \langle \Delta\rangle\}. 
\]

\end{enumerate}
\end{corollary}

\begin{proof} It is enough to check (i) as (ii) follows from (i) immediately. By Lemma~\ref{L:trz}~(ii), every endofunctor in $\langle\Delta\rangle$ is frank. 
By Theorem~\ref{T:semig}, each endofunctor in $\langle \Delta\rangle$ 
fulfills (P) at $a, b$. The conclusion follows from Proposition~\ref{P:bas}(i). 
\end{proof}

\subsection{Products}

We define the finitely supported product of categories in a natural way. Let $C_i$, $i\in I$, be a family of categories. 
Define 
\[
\bigotimes_{I} C_i 
\]
as follows. Objects of this category are of the form 
\[
(c_i)_{i\in K}, 
\]
where $K\subseteq I$ is finite and $c_i\in {\rm ob}(C_i)$. Morphisms are of the form 
\[
(f_i)_{i\in K},
\]
where $K\subseteq I$ is finite and $f_i\in {\rm hom}(c_i, d_i)$, for $c_i, d_i\in {\rm ob}(C_i)$, and we declare the above morphism to be a morphism 
from $(c_i)_{i\in K}$ and $(d_i)_{i\in K}$. To relax the notation, we will write 
\[
(c_i)_{K}\;\hbox{ and }\; (f_i)_{K}
\]
for the object $(c_i)_{i\in K}$ and the morphism $(f_i)_{i\in K}$, respectively.

Assume now we have two families of categories $C_i$ and $D_i$ with $i\in I$. Let $\delta_i\colon C_i\to D_i$ be a functor. 
Define $\otimes_{I} \delta_i \colon \bigotimes_{I} C_i \to \bigotimes_{I} D_i$
by letting 
\[
(\otimes_{I}\delta_i)\big((c_i)_{K}\big)= (\delta_i(c_i))_{K}\;\hbox{ and }\; (\otimes_{I} \delta_i)\big((f_i)_{K}\big) = (\delta_i(f_i))_{K}. 
\]
It is immediate that $\otimes_{i\in I} \delta_i$ is a functor. If $C_i=C$ and $\delta_i=\delta$ for all $i\in I$, we write $\bigotimes_IC$ and $\otimes_I\delta$ 
for $\bigotimes_IC_i$ and $\otimes_I\delta_i$, respectively.

The following lemma is easy to check and we leave doing it to the reader. 

\begin{lemma}\label{L:prfr} 
If each functor $\delta_i\colon C_i\to D_i$, $i\in I$ is frank, 
then $\otimes_{I}\delta_i\colon \bigotimes_{I} C_i\to \bigotimes_{I} D_i$ is frank. 
\end{lemma} 

The following theorem gives the transfer of property (P) from the factors to the product of categories. 

\begin{theorem}\label{T:pro}
Let $\delta_i\colon C_i\to D_i$, $i\in I$ be frank functors. 
Let $K\subseteq I$ be finite, and 
assume that ${\rm hom}(a,b)$ is finite for all $a,b\in {\rm ob}(C_i)$ with $i\in K$. 
If $\delta_i$ fulfills {\rm (P)} at $a_i\in {\rm ob}(C_i)$, for each $i\in K$, then 
$\otimes_{I}\delta_i$ fulfills {\rm (P)} at $(a_i)_{K}$. 
\end{theorem} 

\begin{proof} 
Fix an enumeration of $K$, that is, 
\[
K= \{ i_j\mid j=1, \dots, k\}.
\]
We define the following families of categories 
\[
E^0_i=C_i,\hbox{ for all }i\in I,
\]
and, for $p=1, \dots, k+1$,
\[
E^p_i=
\begin{cases}
D_i, &i\in I\setminus K;\\
D_i, &i= i_j, \hbox{ for some } j<p;\\
C_i, &i=i_j, \hbox{ for some } j\geq p.
\end{cases}
\]
In particular, $E^1_i= D_i$ for all $i\in I\setminus K$, and $E^1_i=C_i$ for all $i\in K$, and $E^{k+1}_i=D_i$ for all $i\in I$. 
Similarly, define $\delta^p_i\colon E^p_i\to E^{p+1}_i$, for $i\in I$ and $p=0, \dots, k$, by letting 
\[
\delta^0_i=
\begin{cases}
\delta_i, &i\in I\setminus K;\\
{\rm id}_{C_i}, &i\in K;
\end{cases}
\]
and, for $p\geq 1$, 
\[
\delta^p_i = 
\begin{cases}
\delta_i, &i = i_p;\\
{\rm id}_{D_i}, & i\in I\setminus K\hbox{ or } i= i_j\hbox{ for some } j< p;\\
{\rm id}_{C_i}, &  i= i_j\hbox{ for some } j> p
\end{cases} 
\]
For $p=0, \dots, k$, consider the functor 
\[
\otimes_I \delta^p_i: \bigotimes_I E^p_i \to \bigotimes_I E^{p+1}_i.
\]
Set $\widehat{\delta^p} = \otimes_I \delta^p_i$, 
and note that, by Lemma~\ref{L:prfr}, each $\widehat{\delta^p}$ is frank and that 
\[
\otimes_I \delta_i = \widehat{\delta^k} \circ \cdots \circ \widehat{\delta^0}.
\]

Fix $a_i\in {\rm ob}(C_i)$, for $i\in K$. 
By Theorem~\ref{T:semig}, it suffices to show that, for each $p\leq k$,  
$\widehat{\delta^p}$ fulfills (P) at 
\begin{equation}\label{E:prob} 
\big( \widehat{\delta^{p-1}} \circ \cdots \circ \widehat{\delta^0}\big) \big( (a_i)_{K}\big)
\;\hbox{ and }\; (b'_i)_{K'},
\end{equation} 
where $(b'_i)_{K'}$ is an arbitrary object of ${\rm ob}( \bigotimes_I E^p_i)$. 
This is clear for $p=0$ and, for arbitrary $p$, when $K'\not= K$ as in this case there are no morphisms between the two objects in \eqref{E:prob}. 
Now consider the case $p\geq 1$ and $K=K'$. 
Set 
\[
(a_i')_K= \big( \widehat{\delta^{p-1}} \circ \cdots \circ \widehat{\delta^0}\big) \big( (a_i)_{K}\big),
\] 
and note that 
\[
a'_{i}= 
\begin{cases}
a_{i},& \text{ if } i=i_j\text{ for }j\geq p;\\
\delta_{i}(a_{i}),& \text{ if } i=i_j\text{ for } j<p; 
\end{cases}
\]
Let  
\[
R= r^M, 
\]
where 
\[
M= \prod_{i\in K, i\not=i_p} |{\rm hom}\big(a'_{i}, b'_{i}\big)|.
\]
By our assumption of finiteness of ${\rm hom}(a, b)$ for $a,b \in {\rm ob}(C_i)$ and frankness of $\delta_i$ in conjunction with Lemma~\ref{L:trz}~(iii), for all $i\in K$, 
we see that $M$, and so also $R$, is finite. 
Find $c\in C_{i_p}$ that witnesses property (P) for the functor $\delta_{i_p}$ at $a_{i_p}$, $b_{i_p}$ with $R$ colors. 

We claim that the object $(c_i)_{K}$ of $\bigotimes_{I} C_i$ with 
\[
c_{i}= 
\begin{cases}
b'_{i},& \text{ if } i\not=i_p;\\
c,& \text{ if }i=i_p;\\
\end{cases}
\]
witnesses property (P) for the functor $\widehat{\delta^{p}}$ at 
$(a'_i)_{K}$ and $(b'_i)_{K}$ with $r$ colors. Indeed, let $\chi$ be an $r$-coloring of ${\rm hom}( (a'_i)_{K}, (c_i)_{K})$. Define an $R$-coloring $\chi'$ of 
${\rm hom}(a_{i_p}, c)$ by letting $\chi'(h)$ be the sequence 
\begin{equation}\label{E:dech}
\big(\chi((f_i)_{K})\mid\; f_{i_p}=h,\; 
f_i\in {\rm hom}(a'_i, b'_i) \hbox{ for }i\not= i_p\big). 
\end{equation}
By our choice of $c$, there is $g\in {\rm hom}(b_{i_p}, c)$ such that for $f,f'\in {\rm hom}(a_{i_p}, b_{i_p})$ we have 
\begin{equation}\label{E:down}
\delta_{i_p}(f)=\delta_{i_p}(f')\Longrightarrow \chi'(g\cdot f) = \chi'(g\cdot f'). 
\end{equation}
Let $(g_i)_{K} \in {\rm hom}((b'_i)_{K}, (c_i)_{K})$ be such that $g_{i_p}= g$ and, for $i\in K,\, i\not= i_p$,  $g_i$ is the identity in 
${\rm hom}(b'_i,c_i)= {\rm hom}(b'_i, b'_i)$. 
It is now easy to check, from \eqref{E:dech} and \eqref{E:down}, that 
for $(f_i)_{K},\, (f'_i)_{K} \in \hom\big((a_i')_K, (b_i')_K\big)$, 
\begin{equation}\notag
\widehat{\delta^p}\big((f_i)_{K}\big) = \widehat{\delta^p}\big((f'_i)_{K}\big) \Longrightarrow 
\chi\big((g_i)_{K}\cdot (f_i)_{K}\big) = \chi\big((g_i)_{K}\cdot (f'_i)_{K}\big), 
\end{equation}
as required. 
(iii) It follows from Theorem~\ref{T:semig} that 
\end{proof}

We give now one concrete application of Theorem~\ref{T:pro}. 
Recall the category $P$  and the functor $\partial_P$ from Section~\ref{Su:exp}. From Lemma~\ref{L:pfr} and Theorem~\ref{T:pro}, we immediately get. 

\begin{corollary}\label{C:prr}
The functor $\otimes_{\mathbb N}\partial_P\colon \bigotimes_{\mathbb N}P\to \bigotimes_{\mathbb N}P$ is frank and fulfills {\rm (P)}. 
\end{corollary}

\subsection{Modeling} 

We present a notion that allows property (P) to be transferred from one functor to another. We aim to give this notion its natural generality. 
Two related notions have already been proposed in the literature. One was the notion of interpretability that 
was defined by the author in \cite[Section~6.2]{Sol}, the other one was the notion of pre-adjunction defined by Ma{\v s}ulovi{\'c} in \cite{Mas}. 
Our notion of modeling generalizes both these definitions. 

Let $C$ and $D$ be categories. We define now the notion of cross-relatedness that 
will play an auxiliary, but important, role in the definition of modeling.

\begin{definition} 
Let $c_1, c_2, c_3\in {\rm ob}(C)$ and $d_1, d_2, d_3\in {\rm ob}(D)$. We say that 
$c_1, c_2, c_3$ and $d_1, d_2, d_3$ are {\bf cross-related} if there are functions
\[
\begin{split}
\phi&\colon {\rm hom}(c_1, c_2)\times {\rm hom}(d_2,d_3)\to {\rm hom}(d_1, d_2),\\
\psi&\colon {\rm hom}(d_2,d_3)\to {\rm hom}(c_2,c_3),\\
\zeta&\colon {\rm hom}(d_1, d_3)\to {\rm hom}(c_1, c_3) 
\end{split}
\]
such that, for $(f,g)\in {\rm hom}(c_1, c_2) \times {\rm hom}(d_2,d_3)$, 
\[
\zeta\big( g\cdot \phi(f,g)\big) = \psi(g)\cdot f. 
\]
\end{definition} 

Intuitively, one can see the notion of cross-relatedness as a way to define composition ``$g\cdot f$" of morphisms $f\in \hom(c_1,c_2)$ by morphisms $g\in \hom(d_2,d_3)$. Of course, literally, 
such composition does not exist as $\hom(c_1,c_2)$ is computed in $C$ while $\hom(d_2,d_3)$ in $D$. But it can be defined in a generalized sense, in fact, in two ways. 
In order to do it, one stipulates that there exist functions 
\[
\phi\colon {\rm hom}(c_1, c_2)\to {\rm hom}(d_1, d_2)\;\hbox{ and }\; 
\psi\colon {\rm hom}(d_2,d_3)\to {\rm hom}(c_2,c_3).
\]
that allow one to compute the composition ``$g\cdot f$", for $f\in \hom(c_1,c_2)$ and $g\in \hom(d_2,d_3)$, in two ways 
\[
g\cdot \phi(f)  \in \hom(d_1,d_3)\;\hbox{ and }\; 
\psi(g)\cdot f \in \hom(c_1,c_3). 
\]
To relate these two results one stipulates further that there is a function 
\[
\zeta\colon \hom(d_1,d_3)\to \hom(c_1,c_3)
\] 
such that 
\[
\zeta\big( g\cdot \phi(f) \big) = \psi(g)\cdot f.
\]
Being cross-related asserts that the above procedure can be implemented additionally allowing $\phi$ to depend on $g$.

Observe that one can formulate the definition of cross-relatedness without invoking $\zeta$, as, to ensure that such a function $\zeta$ exists, 
it suffices to assume that 
$g\cdot \phi(g,f)$ determines $\psi(g,f)\cdot f$, 
that is, that for all $(f,g), (f',g') \in {\rm hom}(c_1, c_2) \times {\rm hom}(d_2,d_3)$, 
\begin{equation}\label{E:subs}
g\cdot \phi(f,g) = g'\cdot \phi(f',g')\Longrightarrow \psi(g)\cdot f= \psi(g')\cdot f'.
\end{equation} 

When it is important to remember how cross-relatedness of the triples $c_1, c_2, c_3$ and $d_1, d_2, d_3$ is witnessed, we say 
that $c_1, c_2, c_3$ and $d_1, d_2, d_3$ are {\bf cross-related by $(\phi, \psi)$} omitting $\zeta$ from the notation for the reasons explained above.

The notion of cross-relatedness defined here is new; but in \cite{Sol}, a version of it with $\phi(f,g)$ depending only on $f$ is 
present implicitly; in \cite{Mas} another version of it is implicit, 
where $\phi$ and $\psi$ are defined globally as functions from $C$ to $D$ and from $D$ to $C$, 
respectively, and $\zeta$ is assumed to be equal to $\psi$. 

Let $\gamma\colon C\to E$ and $\delta \colon D\to F$ be functors, and let $d_1, d_2\in {\rm ob}(D)$ and $c_1, c_2\in {\rm ob}(C)$. 
\begin{definition} 
We say that $\gamma$ at $c_1, c_2$ is {\bf modeled by} $\delta$ at $d_1, d_2$ if,  
for each $d_3\in {\rm ob}(D)$, there is $c_3\in {\rm ob}(C)$ so that 
$c_1, c_2, c_3$ and $d_1, d_2, d_3$ are cross-related by some $(\phi, \psi)$ such that
\begin{equation}\label{E:crtr}
\gamma(f)=\gamma(f')\Longrightarrow \delta\big(\phi(f,g)\big)= \delta\big(\phi(f', g)\big),
\end{equation}
for all $f, f'\in {\rm hom}(c_1, c_2)$ and $g\in {\rm hom}(d_2, d_3)$. 
\end{definition}

Now, we have the main result of this section on transferring property (P) through modeling. 

\begin{theorem}\label{T:mod}
Let $\gamma$ be a functor with domain $C$. Let $a, b\in {\rm ob}(C)$. If $\gamma$ at $a, b$ is modeled by $\delta$ at $d_1,d_2$ with 
$\delta$ fulfilling {\rm (P)} at $d_1, d_2$, then $\gamma$ fulfills {\rm (P)} at $a, b$. 
\end{theorem}

\begin{proof} Fix the number of colors $r\in {\mathbb N}$. Let 
$\delta$ be a functor with domain $D$ such that $\gamma$ at $a,b$ is modeled by $\delta$ at $d_1, d_2 \in {\rm ob}(D)$ with $\delta$ fulfilling (P) at $d_1,d_2$. 
Let $d_3\in {\rm ob}(D)$ witness property (P) for $\delta$ at $d_1, d_2$ with $r$ colors. 
Find $c\in {\rm ob}(C)$ given for $d_3$ by the definition of modeling. So $a,b,c$ and $d_1, d_2, d_3$ are cross-related by a  
pair of functions $(\phi, \psi)$ as in the definition of modeling. 

We claim that $c$ witnesses property (P) for $\gamma$ at $a,b$
with $r$ colors. Let $\chi$ be an $r$-coloring of ${\rm hom}(a,c)$. 
For $f\in {\rm hom}(a,b)$ and $g\in {\rm hom}(d_2, d_3)$, define 
\begin{equation}\label{E:chchp}
\chi'\bigl(g\cdot \phi(f,g)\bigr) =  \chi\big( \psi(g)\cdot f\big) 
\end{equation}
Note that $\chi'$ is well defined since $a,b,c$ and $d_1, d_2, d_3$ are cross-related by $(\phi, \psi)$. 
The function $\chi'$ is defined on a subset of ${\rm hom}(d_1, d_3)$. We extend it to ${\rm hom}(d_1, d_3)$ in an arbitrary way to get 
an $r$-coloring $\chi'$ of ${\rm hom}(d_1, d_3)$.
Now, by our choice of $d_3$, there exists $g\in {\rm hom}(d_2, d_3)$ such that, for each $h, h'\in {\rm hom}(d_1, d_2)$, 
\[
\delta h =\delta h'\Longrightarrow \chi'(g\cdot h) = \chi'(g\cdot h');
\]
in particular, for so chosen $g$, for all $f,f'\in {\rm hom}(a,b)$, we have 
\[
\delta\big(\phi(f,g)\big) =\delta\big(\phi(f',g)\big)\Longrightarrow \chi'\bigl(g\cdot \phi(f,g)\bigr) = \chi'\bigl(g\cdot \phi(f',g)\bigr).
\]
By our choice of $\delta$, whose relationship with $\gamma$ is given by \eqref{E:crtr}, and the definition \eqref{E:chchp} of $\chi'$, the implication above yields, 
for all $f,f'\in {\rm hom}(a,b)$, 
\[
\gamma f=\gamma f'\Longrightarrow \chi\bigl(\psi(g)\cdot f\bigr) = \chi\bigl(\psi(g)\cdot f'\bigr). 
\]
Thus, condition (P) for $\gamma$ at $a,b$ is proved. 
\end{proof}

To transfer bounds on Ramsey degree, only the following version of 
modeling is needed. 
\begin{definition}
Let $c_1, c_2\in {\rm ob}(C)$ and let $d_1, d_2\in {\rm ob}(D)$. We say that $c_1, c_2$ is {\bf R-modeled by} 
$d_1, d_2$ if for each $d_3\in {\rm ob}(D)$ there exits $c_3\in {\rm ob}(C)$ such that $c_1, c_2, c_3$ and $d_1, d_2, d_3$ are cross-related. 
\end{definition}
With the above definition, one proves the following proposition, which strengthens \cite{Mas}. 

\begin{proposition}
Let $C, D$ be categories, and let $a,b\in {\rm ob}(C)$ and $d_1, d_2\in {\rm ob}(D)$. If $a,b$ is R-modeled by $d_1, d_2$, 
then 
\[
{\rm rd}(a,b)\leq {\rm rd}(d_1, d_2).
\] 
\end{proposition}

\begin{proof} This proof is parallel to the proof of Theorem~\ref{T:mod}. 
Let $k= {\rm rd}(d_1, d_2)$. We can assume that $k<\infty$. 
We show that $ {\rm rd}(a,b)\leq k$. Fix $r\in {\mathbb N}$. Let now $d_3\in {\rm ob}(D)$ witness ${\rm rd}(d_1, d_2)\leq k$ for $r$-colorings. 
Let $c\in {\rm ob}(C)$ be provided from the definition of R-modeling so that 
$a,b,c$ and $d_1, d_2, d_3$ are cross-related by $(\phi, \psi)$. 

We claim that $c$ chosen above witnesses ${\rm rd}(a,b)\leq k$ for $r$-colorings. 
To check this claim, let now $\chi$ be an $r$-coloring of ${\rm hom}(a,c)$. Define a coloring $\chi'$ of ${\rm hom}(d_1, d_3)$ exactly 
as is done around formula \eqref{E:chchp} in the proof of Theorem~\ref{T:mod}. Now by our choice of $k$, there is $g\in {\rm hom}(d_2, d_3)$ 
such that $\chi'$ attains at most $k$ colors on the set 
\[
\{ g\cdot f\mid f\in {\rm hom}(d_1, d_2)\}. 
\]
By the definition of $\chi'$, we see that for so chosen $g$, $\chi$ attains at most $k$ colors on the set 
\[
\{ \psi(g)\cdot f\mid f\in {\rm hom}(a,b)\},
\]
as required. 
\end{proof}

\subsection{Example---the Hales--Jewett theorem} 

We fix $k_0\in {\mathbb N}$. Below, for a function $f$, ${\rm im}(f)$ will stand for the set of all values of $f$.

We define a category ${\rm HJ}_{k_0}$ and its endofunctor $\partial_{k_0}$, which will be used to prove the Hales--Jewett theorem.

Objects of ${\rm HJ}_{k_0}$ are as follows: 
\begin{enumerate}
\item[---] natural numbers $l\in {\mathbb N}$;

\item[---] surjections $v\colon [-k_0, 0] \to [k]$, for some $k\in {\mathbb N}$. 
\end{enumerate}
In the remainder of this section, $l$, possibly with subscripts, will stand for objects of ${\rm HJ}_{k_0}$ of the first kind above 
and $v$, possibly with subscripts, will stand for objects of the second kind. 

Morphisms of ${\rm HJ}_{k_0}$ will be appropriate functions. We describe them as follows: 
\begin{enumerate}
\item[---] $f\in {\rm hom}(v, l)$ is a function $f\colon [l]\to {\rm im}(v)$; 
\item[---] $g\in {\rm hom}(l_1, l_2)$ is a function $g\colon [l_2] \to [-k_0, l_1]$ such that ${\rm im}(g)\supseteq [l_1]$; 
\item[---] there are no other morphisms except for the identities. 
\end{enumerate}

In order to define composition of morphisms in ${\rm HJ}_{k_0}$, we need to introduce a new piece of notation. For two functions 
$h$ and $h'$ domains are disjoint intervals $I$ and $I'$ of ${\mathbb Z}$, respectively, let 
\[
h^\frown h'
\]
stand for the function whose domain is $I\cup I'$ and whose restrictions to $I$ and $I'$ are $h$ ad $h'$, respectively. Now 
composition in ${\rm HJ}_{k_0}$ is defined as follows: 
\begin{enumerate}
\item[---] for $f\in {\rm hom}(v, l_1)$ and $g\in {\rm hom}(l_1, l_2)$, let 
\[
g\cdot f = (v^\frown f)\circ g. 
\]
\item[---] for $g_1\in {\rm hom}(l_1, l_2)$ and $g_2\in {\rm hom}(l_2, l_3)$, let  
\[
g_2\cdot g_1 = \big(({\rm id}_{[-k_0,0]})^\frown g_1\big)\circ g_2. 
\]
\end{enumerate} 
Note that $v^\frown f$ and $({\rm id}_{[-k_0,0]})^\frown g_1$, whose domains are $[-k_0, l_1]$ and $[-k_0, l_2]$, respectively.

Before, we define the functor $\partial_{k_0}\colon {\rm HJ}_{k_0}\to {\rm HJ}_{k_0}$, we need to modify \eqref{E:dotmin}. For $n\in {\mathbb N}$, $n\geq 1$, 
let 
\[
n\divdot 1 = \max( n-1, 1), 
\]
The functor $\partial_{k_0}\colon {\rm HJ}_{k_0}\to   {\rm HJ}_{k_0}$ is now defined on objects by:
\begin{enumerate}
\item[---] $\partial_{k_0}(l)=l$;

\item[---] $\partial_{k_0}(v) = \min\big( v,\, \max ({\rm im}(v)) \divdot 1\big)$;
\end{enumerate}
and on morphisms by: 
\begin{enumerate} 
\item[---] $\partial_{k_0}(g)=g$, for $g\in {\rm hom}(l_1, l_2)$; 

\item[---] $\partial_{k_0}(f) =  \min \big( f,\, \max ({\rm im}(v)) \divdot 1 \big)$, for $f\in {\rm hom}(v, l)$. 
\end{enumerate}

\begin{lemma}\label{L:hjt}
\begin{enumerate} 
\item[(i)] $\partial_{k_0}$ is a functor and it is frank. 

\item[(ii)] The functor $\partial_{k_0}: {\rm HJ}_{k_0}\to {\rm HJ}_{k_0}$ fulfills {\rm (P)}. 
\end{enumerate}
\end{lemma}

\begin{proof} (i) is done by an easy check that we leave to the reader.

We now prove (ii). Note that it is clear that $\partial_{k_0}$ fulfills (P) at pairs of objects of the form $(l_1,l_2)$, $(v_1, v_2)$, and $(l,v)$. Indeed, the morphism sets 
between objects in each of these pairs are empty with the exception of $l_1, l_2$ with $l_1\leq l_2$, in which case, $\delta_{k_0}$ is equal to the identity map on 
$\hom (l_1,l_2)$. 

It remains to check that $\partial_{k_0}$ fulfills (P) at pairs of the form $v, l\in {\rm ob}({\rm HJ}_{k_0})$. 
This goal will be achieved by showing that $\partial_{k_0}$ at $v, l$ is modeled by $\otimes_{\mathbb N} \partial$ at a pair of object of $\bigotimes_{\mathbb N}P$ that we will 
choose below, and using Theorem~\ref{T:mod} and Corollary~\ref{C:prr}. 
Let $k_1\in {\mathbb N}$ be such that ${\rm im}(v)= [k_1]$. So, at this point, we have fixed $v, l, k_1$, and, of course, $k_0$.

We define the two objects of $\bigotimes_{\mathbb N} P$ that will be used to model $\partial_{k_0}$ at $v,l$. 
Let $a = (a_i)_{i=1}^l$ and $b= (b_i)_{i=1}^l$, 
where $a_i, b_i\in {\rm ob}(P)$, for $1\leq i\leq l$, be defined by 
\[
a_i= (k_1, 1)\;\hbox{ and }\; b_i=(3,2). 
\]
So $a,b$ are objects in $\bigotimes_{\mathbb N} P$. To see that $\partial_{k_0}$ at $v,l$ is modeled by $\bigotimes_{\mathbb N} \partial_P$ at $a,b$, 
we fix an arbitrary object $c\in {\rm ob}\big(\bigotimes_{\mathbb N}P\big)$. 
We can assume that $c=(c_i)_{i=1}^l$ for some $c_i\in {\rm ob}(P)$ 
with $c_i = (m_i, 2)$ for some $m_i\in {\mathbb N}$, $1\leq i\leq l$. We need to find an object $l'$ in ${\rm HJ}_{k_0}$ 
and functions 
\[
\phi\colon {\rm hom}(v,l)\to {\rm hom}(a,b)\;\hbox{ and }\; \psi\colon {\rm hom}(b,c)\to {\rm hom}\big(l, l'\big)
\]
such that 
\begin{equation}\label{E:conp}
\begin{split}
p\cdot \phi(f)= p'\cdot \phi(f') \Rightarrow &\, \psi(p)\cdot f = \psi(p')\cdot f',\\ 
&\hbox{ for all } p, p'\in {\rm hom}(b,c), \, f, f'\in {\rm hom}(v,l).
\end{split}
\end{equation}
Translating (FP) to the situation dealt with here, the function 
$\phi$ should be defined on ${\rm hom}(v,l)\times {\rm hom}(b,c)$; but our $\phi$ will not depend on the second coordinate.

We can define $\phi$ right away. 
For $f\in {\rm hom}(v,l)$, set 
\[
\phi(f) = (\phi_1(f), \dots , \phi_l(f)), 
\]
where, for each $1\leq i\leq l$, $\phi_i(f): [3]\to [k_1]$ is defined by letting 
\[
\big(\phi_i(f)\big)(j) = 
\begin{cases}
k_1, &\text{ if } j=1;\\
f(i), &\text{ if } j=2;\\
\max(1, k_1-1), &\text{ if }j=3.
\end{cases}
\]
Note that $\phi_i(f)\in {\rm hom}(a_i,b_i)$, and, therefore, $\phi(f) \in {\rm hom}(a,b)$. 

With the definition of $\phi$ in hand, we state \eqref{E:conp} in more basic terms. Note that each $p\in {\rm hom}(b,c)$ is of the form 
\begin{equation}\label{E:pp}
p= (p_1, \dots, p_l),  
\end{equation}
where $p_i$ is in ${\rm hom}(a_i,b_i)$, that is, it is an non-decreasing surjections such that ${\rm im}(p_i) = [3]$ and ${\rm dom}(p_i) = [m_i]$, for each $1\leq i\leq l$. 
Similarly, we represent $p'\in {\rm hom}(b,c)$ as $(p_1', \dots, p_l')$. 
Now, \eqref{E:conp} becomes 
\begin{equation}\label{E:comf} 
\begin{split}
\Bigl( \phi_i(f)\circ p_i = \phi_i(f')\circ p'_i,\hbox{ for }&1\leq i\leq l\Bigr) \Rightarrow f\circ \psi(p) = f'\circ \psi(p')\\
&\hbox{ for all } p, p'\in {\rm hom}(b,c), \, f, f'\in {\rm hom}(v,l).
\end{split}
\end{equation}

It remains to define $l'$ and $\psi$ for which \eqref{E:comf} holds. Set 
\[
l'= m_1+\cdots +m_l.
\]
For $p$ as in \eqref{E:pp}, define 
$g= \psi(p) \in {\rm hom}(l, l')$ as follows. If $j\in [m_1+\cdots + m_l]$, let $i$ be the unique natural number such that 
\[
m_1+\cdots + m_{i-1}< j \leq m_1+\cdots + m_i.
\]
Then if $p_i(j- (m_1+\cdots + m_{i-1})) = 2$, let 
\[
g(j) = i,
\]
otherwise, let $g(j)$ be a number in $[-k_0, 0]$ such that 
\[
v(g(j))=k_1, \hbox{ if }p_i(j- (m_1+\cdots + m_{i-1})) = 1, 
\]
and 
\[
v(g(j))=\max(1, k_1-1), \hbox{ if } p_i(j- (m_1+\cdots + m_{i-1})) = 3. 
\]

With the definitions above, it is easy to check that, for $p\in {\rm hom}(b,c)$ represented as in \eqref{E:pp} and for $f\in {\rm hom}(v,l)$, we have 
\[
v^\frown(\phi_1(f)\circ p_1)^\frown\cdots{}^\frown (\phi_l(f)\circ p_l) = f\circ \psi(p),
\]
from which \eqref{E:comf} follows immediately. 

To finish the proof of (FP), it remains to show that for $f, f'\in {\rm hom}(v, l)$, 
\[
\partial_{k_0} f = \partial_{k_0} f' \Rightarrow \big(\bigotimes_{\mathbb N} \partial_P\big) \big(\phi(f)\big) = \big(\bigotimes_{\mathbb N} \partial_P\big) \big(\phi(f')\big), 
\]
which amount to proving 
\[
\partial_{k_0} f = \partial_{k_0} f' \Rightarrow \partial_P \big(\phi_i(f)\big) =  \partial_P\big(\phi_i(f')\big),\;\hbox{ for }i= 1, \dots, l. 
\]
\end{proof} 

The following is the Hales--Jewett theorem, see \cite{Nes}. 

\begin{corollary}\label{C:hj} 
For each $k, l\in {\mathbb N}$ and $r\in {\mathbb N}$, there exists $m\in {\mathbb N}$ such that for each 
$r$-coloring of all functions from $[m]$ to $[-k, 0]$, there 
is $g\colon [m]\to [-k, l]$, with ${\rm im}(g)\supseteq [l]$, such that the set 
\[
\{ f\circ g\mid f\colon [-k,l]\to [-k,0], \, f\res [-k,0]= {\rm id}_{[-k,0]}\} 
\]
is monochromatic. 
\end{corollary}

\begin{proof} Fix $k, l\in {\mathbb N}$. Consider the category ${\rm HJ}_k$ and the objects ${\rm id}_k$ and $l$ in it. Note that the conclusion of the corollary follows from 
${\rm rd}({\rm id}_k, l)=1$. To prove this equality, observe that the set 
\[
(\partial_k\circ \cdots \circ \partial_k)\big( {\rm hom}({\rm id}_k, l)\big),
\]
where $\partial_k$ is composed $k-1$ times, has one element. 
By Lemma~\ref{L:hjt} combined with Corollary~\ref{C:pest}(i) for $\Delta=\{ \partial_k\}$, the equality ${\rm rd}({\rm id}_k, l)=1$ follows immediately. 
\end{proof}

With some additional routine work, the methods used to prove Corollary~\ref{C:hj} can be adapted to proving more general versions of the Hales--Jewett theorem as in 
\cite[Section~8.1]{Sol} or \cite[Lemma~3.3]{Solt}. In these generalizations, one obtains concrete Ramsey statements, in which ${\rm rd}(a,b)$ may be strictly bigger than $1$.

\section{Proving condition (P)}\label{S:prp}

The goal of this section is to formulate a local version of condition (P) and prove that, in most circumstances, it implies (P).

\subsection{Condition (FP)}

We state the local version, we call (FP), of the pigeonhole principle (P). In applications, it is often easier to check directly (FP) than (P). 
Let $\delta\colon C\to D$ be a functor. Recall the statement of (P) for a functor $\delta$ from Definition~\ref{D:defp}. 
Note that the property of $g\in \hom(b,c)$ in condition (P) can be rephrased as follows: 

\smallskip

\centerline{\em for each $f'\in \delta\big(\hom(a, b)\big)$, $\chi(g\scirc f)$ is constant for $f\in \hom(a,b)$ with $\delta f = f'$.} 

\smallskip

\noindent Above, $g$ is chosen first and independently of $f'$. This feature is relaxed when passing to (FP) from (P), namely, in (FP) 
it suffices to find $g$ that depends on $f'$. 
The price of this relaxation is included in the second point of condition (FP). It has to do with controlling the behavior of $\delta g$ in a suitable way. 

It will be convenient to introduce the following 
piece of notation. Let $C$ be a category $C$ and $\delta$ a functor defined on $C$. For $a,b\in {\rm ob}(C)$ and $h\in \delta\big( {\rm hom}(a,b)\big)$, let 
\[
{\rm hom}(a,b)_h =\{ f\in {\rm hom}(a,b)\mid \delta f =h\}. 
\]

\begin{definition}
We say that $\delta$ {\bf fulfills (FP) at} $a, b\in {\rm ob}(C)$, if 
for $r\in {\mathbb N}$ and a finite non-empty set $s\subseteq \delta\big( \hom( a, b)\big)$ the following condition holds:

\noindent there exist $c\in {\rm ob}(C)$, $f'\in s$, and $g'\in \delta\big( \hom( b, c)\big)$ such that for each $r$-coloring $\chi$ of $\hom(a,c)$, there
exists $g\in \hom(b,c)$ with 
\begin{enumerate}
\item[---] $g\scirc {\rm hom}(a,b)_{f'}$ $\chi$-monochromatic and 

\item[---] $(\delta g)\scirc e = g'\scirc e$, for each $e\in s$.
\end{enumerate}
\end{definition} 

As with condition (P), we say that $\delta$ {\bf fulfills (FP) at} $a\in {\rm ob}(C)$, if it fulfills {\rm (FP)} at $a,b$ for all $b\in {\rm ob}(C)$; and we simply say 
that $\delta$  {\bf fulfills (FP)}, if it fulfills {\rm (FP)} at $a,b$ for all $a, b\in {\rm ob}(C)$.

\subsection{Example---a frank functor with (FP)} 
Recall Example 3 from Section~\ref{Su:exp}. We check condition (FP) for $\partial_R$ from this example.

\begin{lemma}\label{L:ramram} 
$\partial_R\colon R\to R$ fulfills $\rm (FP)$. 
\end{lemma} 

\begin{proof} This proof amounts to an application of the standard pigeonhole principle. Fix $r\in {\mathbb N}$. Let 
$k, l$ be two objects in $R$, and let $\emptyset\not= s\subseteq \partial_R\big( {\rm hom}( k, l)\big)$. To avoid trivial cases, 
we can assume that ${\rm hom}\big( k, l\big)$ has at least two elements and that $k\geq 1$, therefore, $1\leq k< l$. 

For the two objects $k$ and $l$ and the number of colors $r$, we need to find, in the notation of (FP), 
and object $c$ and two morphisms $f'$ and $g'$. First, we define the object by letting 
\[
m=(r+1)l \in {\rm ob}(R).
\]
Next, we define the two objects. Pick $(x', l-1)\in \partial_R\big( {\rm hom}( k, l)\big)$ so that  
\begin{equation}\label{E:ob1}
(x',l-1)\in s\;\hbox{ and }\,\max x' = \max\{ \max x''\mid (x'', l-1)\in s\}, 
\end{equation}
and let 
\begin{equation}\label{E:ob2}
(y', m-1)=([l-1], m-1) \in \partial_R\big( {\rm hom}( l, m)\big).
\end{equation} 
By convention, if $k=1$, we interpret the above definition to give $x'=\emptyset$. 
We claim that this choice of the object $m$ and the morphisms $(x',l-1)$ and $(y', m-1)$ ensures that (FP) 
are satisfied. 

To prove this claim, let $\chi$ be an $r$-coloring of ${\rm hom}(k, m)$. 
For $i\in [l,m]$, set 
\[
{\bf x}_i = (x'\cup\{ i\}, m)\in {\rm hom}(k,m), 
\]
and consider the $r$-coloring of $[l,m]$ given by 
\begin{equation}\label{E:cou}
[l,m]\ni i\to  \chi({\bf x}_i). 
\end{equation} 
Set 
\[
p =\max x'\leq l-1,
\]
with $p=0$, if $x'=\emptyset$, by convention. Note that, by the choice of $m$, there is a subset $I$ of $[l,m]$ of size $l-p$ on which the $r$-coloring \eqref{E:cou} is constant, which 
means that $\chi$ is constant on ${\bf x}_i$ as $i$ varies over $I$. 
Define 
\[
{\bf y}= ( [p] \cup I,m)\in {\rm hom}(l, m).
\]
This is the morphism $g$ in the notation from (FP). We need to check that $\bf y$ satisfies the two points displayed in (FP).  

For $(x,l) \in {\rm hom}(k,l)$ with $\partial_R (x,l) = (x', l-1)$, we have 
\[
{\bf y}\cdot (x, l)= {\bf x}_i,\hbox{ for some }i\in I.
\]
Therefore, $\chi({\bf y}\cdot (x,l))$ is constant for $(x, l)\in {\rm hom}(k,l)$ with $\partial_R (x, l) = (x', l-1)$, 
and the first point in (FP) is checked for $\bf y$. 
To see the second point, note that, for each $(x'', l-1)\in s$, we have 
\[
\partial_R {\bf y} \cdot (x'', l-1) = (x'', m-1)  = (y', m-1)\cdot (x'', l-1), 
\]
using \eqref{E:ob1} the get the first equality and \eqref{E:ob1} and \eqref{E:ob2} to get the second one. 
Thus, (FP) follows. 
\end{proof}

\subsection{Condition (FP) implies (P)} 

The following theorem is the main result of Section~\ref{S:prp}. It shows that under mild assumptions the local condition (FP) implies (P). 
In concrete situations, (FP) is usually much easier to check than (P). 

\begin{theorem}\label{T:lpp}
Let $\delta\colon C\to D$ be a functor, and let $a\in {\rm ob}(C)$. If $\delta$ fulfills $\rm (FP)$ at $a$, then 
$\delta$ fulfills $\rm (P)$ at $a,\, b$ for each $b\in {\rm ob}(C)$ with $\delta\big( \hom (a, b)\big)$ finite. 
\end{theorem}

\begin{proof} Fix a functor $\delta$ and an object $a$. In order to prove that $\delta$ fulfills (P) at $a,b$, we 
fix $r>0$ and $b\in {\rm ob}(C)$ with finite $\delta\big( \hom( a, b)\big)$. Set $n=|\delta\big( \hom(a, b)\big)|$. 
By recursion, we construct 
\begin{enumerate}
\item[---] $c_k\in {\rm ob}(C)$, for $0\leq k\leq n$; 

\item[---] $g_k'\in \delta\big( \hom(c_{k-1}, c_k)\big)$, for $1\leq k\leq n$; 

\item[---] $f_k'\in \delta\big( \hom(a, b)\big)$, for $1\leq k\leq n$.
\end{enumerate}
Note that we enumerate the $c_k$-s starting with $k=0$ and the $g'_k$-s and $f'_k$-s starting with $k=1$. 
These objects will have the following properties for $0\leq k\leq n$, where we note that the first point in (c) makes sense 
as, by the conditions above, we have that $g'_{k-1}\cdots g'_1\cdot f'_k\in \delta\big( {\rm hom}(a, c_{k-1})\big)$.
\begin{enumerate}
\item[(a)] $c_0=b$;

\item[(b)] $f'_k \not= f'_i$ for all $i<k$; 

\item[(c)] for each $r$-coloring $\chi$ of $\hom(a, c_k)$ there exists $g\in \hom(c_{k-1}, c_k)$ such that 
\begin{enumerate}
\item[---] $\chi$ is constant on $g\cdot \bigl( \hom(a, c_{k-1})_{g_{k-1}'\cdot \cdot \cdot g'_1\cdot f_k'}\bigr)$; 

\item[---] $(\delta g) \cdot g'_{k-1}\cdot \cdot \cdot g'_1\cdot f' =  g'_k\cdot g'_{k-1}\cdot \cdot \cdot g'_1\cdot f'$ for all $f'\in \delta\big( \hom(a,b)\big)$ with $f'\not\in \{ f_i\mid i\leq k\}$. 
\end{enumerate}
\end{enumerate}
To start the construction, we set $c_0=b$. The conditions above for $k=0$ hold with (b) and (c) being vacuously true. 
Assume that $0<k\leq n$ and the construction has been carried out up to stage $k-1$. Consider the finite set 
\[
s= \{ g_{k-1}'\cdot \cdot \cdot g'_1\cdot f' \mid f'\in \delta\big( \hom(a, b)\big) \setminus \{ f'_i\mid i<k\}\}. 
\]
By our choice of $g'_i$ for $1\leq i<k$ and the assumption that $k\leq n$, we have that 
\[
\emptyset \not= s\subseteq \delta\big( \hom(a, c_{k-1})\big). 
\]
Now condition $\rm (FP)$ applied to $a,\, c_{k-1}$, and the set $s$ above allows us to pick $c_k\in {\rm ob}(C)$, $f'_k \in \delta\big( \hom(a,b)\big)$, and 
$g'_k\in \delta\big( \hom(c_{k-1}, c_k)\big)$ so that conditions (b) and (c) hold. 
The construction has been carried out. 

Observe that by the choice of $n$ and condition (b), we have 
\begin{equation}\label{E:exh}
\delta\big( \hom(a,b)\big) = \{ f_k'\mid 1\leq k\leq n\}. 
\end{equation}

We claim that $c=c_n$
witnesses that $\delta$ fulfills $\rm (P)$ at $a,\, b$ with $r$ colors; that is, 
for each $r$-coloring $\chi$ of $\hom(a,c)$
there is $g\in \hom(b,c)$ such that, for $h_1, h_2\in \hom(a,b)$, we have
\begin{equation}\label{E:dei}
\delta h_1= \delta h_2\Longrightarrow \chi(g\cdot h_1)= \chi(g\cdot h_2).
\end{equation}

In order to prove the statement above, fix an $r$-coloring $\chi$ of $\hom(a,c)$. We recursively
produce 
\[
g_n\in \hom(c_{n-1}, c_n),\, \dots ,\, g_1\in \hom(c_0, c_1)
\] 
starting with $g_n$ and ending with $g_1$ as follows. Having produced $g_n, \dots, g_{k+1}$, we consider the $r$-coloring of $\hom(a, c_k)$ given by
\[
\hom(a, c_k)\ni f \to \chi(g_n\cdot \cdot \cdot g_{k+1}\cdot f).
\]
By (c), we get $g_k\in \hom(c_{k-1}, c_k)$ such that
\begin{equation}\label{E:coco}
\chi(g_n\cdot \cdot \cdot g_{k+1} \cdot g_k \cdot f) \hbox{ is constant for }f\in \hom(a, c_{k-1})_{g_{k-1}'\cdot \cdot \cdot g'_1\cdot f_k'}
\end{equation}
and 
\begin{equation}\label{E:cucu}
(\delta g_k) \cdot g'_{k-1}\cdot \cdot \cdot g'_1\cdot f'_j =  g'_k\cdot g'_{k-1}\cdot \cdot \cdot g'_1\cdot f'_j,\, \hbox{ for }j> k.
\end{equation}

Now we show that 
\[
g = g_n\cdot \cdot \cdot g_1\in \hom(b,c) 
\] 
witnesses that the implication in \eqref{E:dei} holds. Let $h_1, h_2\in \hom(a,b)$ be such that
$\delta h_1 =\delta h_2$. This common value can be taken to be $f'_k$ for some $1\leq k\leq n$ by \eqref{E:exh}. 
For $i=1,2$, an iterative application of condition \eqref{E:cucu} gives
\begin{equation}\notag
\begin{split}
g'_{k-1}\cdot g'_{k-2}\cdot \cdot \cdot g'_{1}\cdot f'_k &= \delta g_{k-1} \cdot \delta g_{k-2} \cdot \cdot \cdot \delta g_{1} \cdot \delta h_i\\
&= \delta (g_{k-1}\cdot g_{k-2}\cdot \cdot \cdot g_{1}\cdot h_i), 
\end{split} 
\end{equation}
and so
\[
g_{k-1} \cdot g_{k-2}\cdot \cdot \cdot g_{1}\cdot h_i \in (\hom(a, c_{k-1}))_{g_{k-1}'\cdot \cdot \cdot g'_1\cdot f_k'},
\]
which in light of \eqref{E:coco} implies that
\begin{equation}\notag
\chi\bigl(g_n\cdot \cdot  \cdot g_k \cdot (g_{k-1}\cdot \cdot \cdot g_1 \cdot h_1)\bigr)= 
\chi\bigl(g_n\cdot \cdot \cdot g_{k} \cdot (g_{k-1}\cdot \cdot \cdot g_1 \cdot h_2)\bigr). 
\end{equation}
Thus, \eqref{E:dei} is proved. 
\end{proof}

The following corollary follows immediately from Corollary~\ref{C:pest}(i) and Theorem~\ref{T:lpp}. 

\begin{corollary}\label{C:lpes} 
Let $\Delta$ be a family of frank endofunctors of $C$. Let $a\in {\rm ob}(C)$. Assume that each $\delta\in \Delta$ fulfills {\rm (FP)} at $a$ 
and ${\rm hom}(a,b)$ is finite for all $b\in {\rm ob}(C)$. Then, for all $b\in {\rm ob}(C)$, 
\[
{\rm rd}(a,b) \leq \min \{ \big|\overline{\delta}\big({\rm hom}(a,b)\big)\big|\mid \overline{\delta}\in \langle \Delta\rangle\}. 
\]
\end{corollary}

As an illustration, we derive now the classical Ramsey theorem and the product Ramsey theorem from the general results established earlier. 

\begin{corollary}
Given $k, l\in {\mathbb N}$, for each $r\in {\mathbb N}$, there exists 
$m\in {\mathbb N}$ such that, for each $r$-coloring of all $k$-element subsets of $[m]$, there 
exists $b\subseteq [m]$ of size $l$ such that all $k$-element subsets of $b$ get the same color. 
\end{corollary} 

\begin{proof} Since, by Lemma~\ref{L:ramram}, $\partial_R$ fulfills (FP), and ${\rm hom}(k,l)$ is finite, for all $k,l\in {\rm ob}(R)$, 
it follows from that $\Delta = \{ \partial_R\}$ satisfies all the assumptions of Corollary~\ref{C:lpes}. Thus, we get the conclusion after noticing that 
the set 
\[
(\partial_R\circ \cdots \circ \partial_R) \big( {\rm hom}(k,l)\big)
\]
has only one element (its only element is the empty set), where $\partial_R$ is composed $k$ times. 
\end{proof}

\begin{corollary}\label{C:prodram} 
Let $r\in {\mathbb N}$ and let $k_1, \dots, k_l$ and $p_1, \dots, p_l$ be natural numbers. There exist natural numbers $q_1, \dots, q_l$ 
such that for each $r$-coloring of the set
\[
\{ (a_1, \dots, a_l)\mid a_i\subseteq [q_i], \, |a_i| = k_i, \hbox{ for }i\leq l\}
\]
there exist $b_1\subseteq [q_1], \dots, b_l\subseteq [q_l]$ with $|b_i|=p_i$, for each $i\leq l$, and such that
the set
\[
\{ (a_1, \dots, a_l)\mid a_i\subseteq b_i, \, |a_i| = k_i, \hbox{ for }i\leq l\}
\]
is monochromatic.
\end{corollary} 

\begin{proof}
By Lemma~\ref{L:ramram}, $\partial_R$ fulfills (FR). Since ${\rm hom}(k,l)$ is finite, for all $k,l\in {\rm ob}(R)$, Theorem~\ref{T:lpp} implies 
that $\partial_R$ fulfills (P).  Thus, by Theorem~\ref{T:pro}, $\otimes_{\mathbb N} \partial_R \colon \bigotimes_{\mathbb N} R\to \bigotimes_{\mathbb N} R$ fulfills (P). 
Now, the conclusion follows from Corollary~\ref{C:pest}(i). 
\end{proof}

\subsection{Example---Fouch{\'e}'s Ramsey theorem for trees}

We present here one more elaborate example of using the general theory. We derive from it Fouch{\'e}'s Ramsey theorem for trees as proved in \cite{Fou}.

First, we collect basic definitions concerning trees and a type of morphism between them. By a {\bf tree}
we understand a {\em finite, non-empty} partial order such that each two elements have a common predecessor and
the set of predecessors of each element is linearly ordered. 
A {\bf leaf} is a maximal element of a tree.
By convention, we regard every node of a tree as one of its
own predecessors and as one of its own successors.

Each tree $T$ carries a binary function $\wedge_T$ that assigns to each
$v,w\in T$ the largest element $v\wedge_Tw$ of $T$ that is a predecessor of
both $v$ and $w$. 
For a tree $T$ and $v\in T$, let
\[
{\rm im}_T(v)
\]
be the set of all {\bf immediate successors} of $v$, and
we do not regard $v$ as one of them.
Let
\[
{\rm ht}_T(v)
\]
be the cardinality of the set of all predecessors of $v$ (including $v$), and let
\[
{\rm ht}(T) = \max \{ {\rm ht}_T(v)\colon v\in T\}.
\]
For a tree $T$, let
\[
{\rm br}(T)
\]
be the maximum of cardinalities of ${\rm im}_T(v)$ for $v\in T$.

A tree $T$ is called {\bf ordered} if for each $v\in T$ there is a fixed linear order of ${\rm im}_T(v)$. Such an assignment allows us to define
the lexicographic linear order $\leq_T$ on all the nodes of $T$ by specifying that $v\leq_T w$ precisely when  
\begin{enumerate} 
\item[---] $v$ is a predecessor of $w$, or 

\item[---] $v$ is not a predecessor of $w$, and $w$ is not a predecessor of $v$, and 
the predecessor of $v$ in ${\rm im}_T(v\wedge w)$ is less than or equal to the predecessor of $w$ in
${\rm im}_T(v\wedge w)$ in the given order on ${\rm im}_T(v\wedge w)$.
\end{enumerate} 

A {\bf height preserving embedding} $f$ from an ordered tree $S$ to an ordered tree $T$ is an injective function 
$f\colon S\to T$ such that 
\begin{enumerate} 
\item[---] $f$ is order preserving between $\leq_S$ and $\leq_T$,

\item[---] $f(v\wedge_Sw) = f(v)\wedge_T f(w)$, for $v,w\in S$, and 

\item[---] ${\rm ht}_S(v) = {\rm ht}_T(f(v))$, for $v\in S$.
\end{enumerate} 
Note that preservation of order by $f$ is equivalent to saying that
for every $v\in S$ and all $w_1,w_2\in {\rm im}_S(v)$
with $w_1\leq _S w_2$, we have $f(w_1)\leq_T f(w_2)$ in ${\rm im}_T(f(v))$.

\begin{theorem}[Fouch{\'e}~\cite{Fou}]\label{T:fou} Let $r\in {\mathbb N}$ and let $S$ and $T$ be ordered trees. There is an ordered tree $V$
such that ${\rm ht}(V) = {\rm ht}(T)$ and for each $r$-coloring of all height preserving embeddings from $S$ to $V$
there is a height preserving embedding $g\colon T\to V$ such that the set
\[
\{ g\circ f\colon f\hbox{ a height preserving embedding of }S\hbox{ to }T\}
\]
is monochromatic.
\end{theorem}

We define a category and an endofuctor on it that are appropriate for the theorem above. 
Consider the category $\mathcal T$ whose objects are ordered trees.
Given $S,T\in {\rm ob}({\mathcal T})$, 
with ${\rm ht}(S) = {\rm ht}(T)$, ${\rm hom}(S,T)$ consists of all height preserving embeddings from $S$ to $T$. There are no other 
morphisms in $\mathcal T$; in particular, if ${\rm ht}(S)\not= {\rm ht}(T)$, then ${\rm hom}(S,T)=\emptyset$. 

We now define a functor $\partial^*\colon {\mathcal T}\to {\mathcal T}$. 
Given $T\in {\rm ob}({\mathcal T})$, put
\begin{equation}\notag
\partial^* T =
\begin{cases}
\{ v\in T\colon {\rm ht}(v)<{\rm ht}(T)\},& \text{ if ${\rm ht}(T)>1$;}\\
T,& \text{ if ${\rm ht}(T)=1$.}
\end{cases}
\end{equation}
We will write $T^*$ for $\partial^* T$. 
Now, define the functor $\partial^*$ on morphisms of $\mathcal T$ by letting, for $f\colon S\to T$,
\[
\partial^* f =  f\res S^*.
\]

\begin{lemma}\label{L:stfr} 
$\partial^*$ is a frank functor.
\end{lemma} 

\begin{proof} 
It is clear that $\partial^*$ is a functor as for morphisms $f\colon S\to T$ and $g\colon T\to V$ we have $f(S^*)\subseteq T^*$ and hence 
\[
\partial^*( g\circ f) = (g\circ f) \res S^* = g\circ (f\res S^*) = (g\res T^*)\circ (f\res S^*) = (\partial^* g)\circ (\partial^* f). 
\]

Now, to check that $\partial^*$ is frank, we fix two objects of $\mathcal T$, that is, two ordered trees $S$ and $T'$. 
We need to find $T\in {\rm ob}({\mathcal T})$ such that $T^*=T'$ and $\partial^*\big( {\rm hom}(S,T)\big) = {\rm hom}(S^*, T^*)$. 
If ${\rm ht}(S) \not= {\rm ht}(T')+1$, then any $T\in {\mathcal T}$ with $T^*= T'$ works, since then ${\rm hom}(S^*, T^*)=\emptyset$ and 
${\rm hom}(S,T)=\emptyset$. So we assume that ${\rm ht}(S) = {\rm ht}(T')+1$. We need to find $T\in {\rm ob}({\mathcal T})$ such that 
\begin{enumerate}
\item[---] $T^*= T'$ and

\item[---] for each height preserving embedding $f'\colon S^*\to T^*$, there is a height preserving embedding $f\colon S\to T$
such that $f' =f\res S^*$. 
\end{enumerate}
One defines $T$ so that ${\rm ht}(T) = {\rm ht}(S)$, $T^*=T'$, and, for each leaf $w$ of $T'$, 
\[
|{\rm im}_T(w)| = {\rm br}(S). 
\]
One then linearly orders $T$ by extending the linear order on $T'$ in an arbitrary way as long as the resulting order makes $T$ into an ordered tree. 
It is then clear that each height preserving embedding $f'\colon S^*\to T'$ extends to a height preserving embedding $f\colon S\to T$ 
by mapping elements of ${\rm im}_S(v)$, for each leaf $v$ of $S^*$, to ${\rm im}_T(f'(v))$ in an injective and order preserving fashion. 
\end{proof}

\begin{proof}[Proof of Theorem~\ref{T:fou}] 
To obtain the conclusion of the theorem, one needs to check that 
${\rm rd}(S,T)\leq 1$ for all $S,T\in {\rm ob}({\mathcal T})$. Note that the set 
\[
(\partial^*\circ\cdots \circ \partial^*)\big( {\rm hom}(S,T)\big) 
\]
has at most one element, where $\partial^*$ is composed ${\rm ht}(S)-1$ many times. Indeed, if ${\rm ht}(S)\not= {\rm ht}(T)$, then ${\rm hom}(S,T)=\emptyset$;
if ${\rm ht}(S)={\rm ht}(T)$, then this set contains only the unique function from a one-node tree to a one-node tree. 
Thus, by Lemma~\ref{L:stfr} and Corollary~\ref{C:lpes}, it will suffice to check that $\partial^*$ fulfills (FP) at each pair of objects of $\mathcal T$. 
So fix $r\in {\mathbb N}$, $S,T\in {\rm ob}({\mathcal T})$, and $\emptyset\not= s\subseteq {\rm hom}(S^*,T^*)$. Non-emptiness of $s$ implies that 
${\rm ht}(S)= {\rm ht}(T)$; we call this common height $h$. We need to produce $V\in {\rm ob}({\mathcal T})$ with ${\rm ht}(V)=h$ and 
$f'\in s$ and $g'\in {\rm hom}(T^*, V^*)$ so that the conclusion of (FP) holds for these choices. 

We let $f'\in s$ be arbitrary and $g'$ be equal to the identity map on $T^*$. 
The tree $V$ will be chosen so that $V^*=T^*$. It suffices to specify, for each leaf $w$ of $T^*$ 
with ${\rm ht}_{T^*}(w)=h-1$, the number of elements in ${\rm im}_V(w)$. If $w$ is not of he form $f'(v)$ for a leaf $v$ of $S^*$, 
let ${\rm im}_V(w)$ be empty. Now, let $v_1, \dots, v_l$ list all the leaves of $S^*$ of height $h-1$; since ${\rm ht}(S)=h$ such leaves exist. Put 
\[
K_i={\rm im}_S(v_i) \;\hbox{ and }\; k_i = |K_i| ,\hbox{ for } 1\leq i\leq l,
\]
and also 
\[
P_i = {\rm im}_T\big(f'(v_i)\big) \;\hbox{ and }\; p_i = |P_i|,\hbox{ for } 1\leq i\leq l. 
\]
Let the natural numbers $q_i$, for $1\leq i\leq l$, be gotten from Corollary~\ref{C:prodram} for the sequences $(k_i)$ and $(p_i)$, 
and the number of colors $r$. Let $Q_i= {\rm im}_V\big(f'(v_i)\big)$ have size $q_i$. This procedure defines $V$. 

Now, it is enough to do the following: 
for an $r$-coloring $\chi$ of ${\rm hom}(S, V)$, find a height preserving embedding $g\colon T\to V$ such that 
\begin{enumerate}
\item[(i)] $\chi(g\circ f)$ is constant on the set of height preserving embeddings $f\colon S\to T$ with $f\res S^* = f'$; 

\item[(ii)] $g\res T^*= {\rm id}_{T^*}$.
\end{enumerate}
We fix $\chi$ as above. For a tuple $(a_i)$ such that 
$a_i\subseteq Q_i$ and $|a_i|=k_i$, for each $1\leq i\leq l$, define $f_{(a_i)}\colon S \to V$  by letting 
\[
\begin{split} 
&f_{(a_i)}\res S^* = f';\\
&f_{(a_i)}\res K_i\colon K_i\to a_i\hbox{ the unique order preserving function.}
\end{split}
\]
Note that $f_{(a_i)}$ is a height preserving embedding. Further note that 
\begin{equation}\label{E:cofg} 
\begin{split}
\Big( f\in {\rm hom}(S, T),\, f\res S^* =f'\hbox{ and }\, &g\in {\rm hom}(T, V),\,g\res T^*={\rm id}_{T^*} \Big) \Longrightarrow\\
&\Big(g\circ f = f_{(a_i)}, \hbox{ where }a_i = g\big(f(K_i)\big)\Big).
\end{split}
\end{equation} 

Color tuples $(a_i)$ such that 
$a_i\subseteq Q_i$ and $|a_i|=k_i$ by letting 
\[
\chi'\big((a_i)\big) = \chi(f_{(a_i)}). 
\]
By our choice of of $(q_i)$, there are sets $b_i\subseteq Q_i$ with $|b_i|= p_i$ and such that all tuples $(a_i)$ with $a_i\subseteq b_i$ 
get the same color with respect to $\chi'$. Let $g\colon T\to V$ be defined by letting 
\[
\begin{split}
&g\res T^*={\rm id}_{T^*};\\
&g\res Q_i\colon Q_i\to b_i \hbox{  the unique order preserving function.}
\end{split}
\]
By \eqref{E:cofg}, for each $f\colon S\to T$ with $f\res S^* = f'$, we have that 
\[
g\circ f = f_{(a_i)},
\]
for some $(a_i)$ with $a_i\subseteq b_i$ and $|a_i|=k_i$. 
Thus, $\chi(g\circ f)$ is constant, as required by (i). Obviously, $g$ fulfills (ii). 
\end{proof}

\end{document}